\documentclass[
final
 , nomarks
]{dmtcs-episciences}

\usepackage[utf8]{inputenc}
\usepackage{subfigure}

\usepackage{amsmath}
\usepackage{amssymb}
\usepackage{amsthm}

\newtheorem{theorem}{Theorem}
\newtheorem{lemma}[theorem]{Lemma}

\theoremstyle{definition}
\newtheorem{defn}{Definition}
\newtheorem{example}{Example}

\author{Adam Borchert
 \and Narad Rampersad\thanks{The second author is supported by an
  NSERC Discovery Grant.}}
\title[Permutation complexity of images of Sturmian words]{Permutation
  complexity of images of Sturmian words by marked morphisms}
\affiliation{
  Department of Mathematics and Statistics,
  University of Winnipeg,
  CANADA}
\received{2017-11-2}
\revised{2018-4-3}
\accepted{2018-5-21}

\begin{document}
\publicationdetails{20}{2018}{1}{20}{4042}
\maketitle
\begin{abstract}
We show that the permutation complexity of the image of a Sturmian
word by a binary marked morphism is $n+k$ for some constant $k$ and
all lengths $n$ sufficiently large.
\end{abstract}

\section{Introduction}
The permutation complexity of an infinite aperiodic word is a concept
introduced by Makarov \cite{Mak06}.  It is based on the following
idea:  Given an infinite word $\omega$, consider the linear order
$\pi_\omega$ on $\mathbb{N}$ induced by the lexicographic order on the
successive shifts of $\omega$.  The permutation complexity of $\omega$
is the function that counts the number of distinct subpermutations of
$\pi_\omega$ of a given length.  Makarov \cite{Mak09} proved that any
Sturmian word $s$ has $n$ subpermutations of length $n$ for all $n
\geq 1$.  In this paper, we determine the permutation complexity of
any word $T(s)$, where $s$ is a Sturmian word, and $T$ is a marked binary
morphism (``marked'' means that the images of the morphism on letters
begin with different letters and end with different letters).

In this paper, we only consider infinite permutations obtained from
infinite words in the manner described above, but
there is also a more general theory of infinite permutations \cite{FF07}.
Avgustinovich, Frid, and Puzynina \cite{AFP17} studied a subclass of
these infinite permutations called \emph{equidistributed
  permutations} and showed that within this family, the infinite permutations
of minimal permutation complexity are exactly those obtained from
Sturmian words.

  There have been several other recent results on permutation
complexity of infinite words.  Here we mention only Widmer's work
\cite{Wid11}, in which he computes the permutation complexity function
of the Thue--Morse word---this turns out to be a rather non-trivial
task---and Valyuzhenich's work \cite{Val14}, which generalizes this
result somewhat.  We should point out that while it may seem rather
unsatisfying to report a result that only applies to marked morphisms,
it appears to be rather difficult to deal with arbitrary morphisms:
in Valyuzhenich's work, he also restricts his attention to marked
morphisms, and even in this case the proofs of his results are quite
difficult.

\section{Preliminaries}
Given an ordered alphabet $\Sigma$, the
\emph{lexicographic order} on $\Sigma^*$ is the order defined as
follows: $u \leq v$ if either
\begin{itemize}
\item $u$ is a prefix of $v$, or
\item $u=xay$, $v=xbz$ for some words $x,y,z$ and letters $a<b$.
\end{itemize}
We write $u<v$ if $u \leq v$ and $u \neq v$.

Let $\omega = \omega_0\omega_1\omega_2\cdots$ be an infinite,
aperiodic word over the alphabet $\{0,1\}$ (throughout this paper all
words will be binary).  We denote the $i$-th letter, $\omega_i$, by
$\omega[i]$, and the factor $\omega_i\omega_{i+1}\cdots\omega_j$ by
$\omega[i,j]$.  The $i$-th shift of $\omega$ is the infinite word
$\omega[i,\infty] = \omega_i\omega_{i+1}\omega_{i+2}\cdots$.  Let the shifts
of $\omega$ be ordered lexicographically (with respect to the order
$0<1$).  Let $\pi_\omega$ be the order on $\mathbb{N}$ defined by
$\pi_\omega(i) < \pi_\omega(j)$ if $\omega[i,\infty] < \omega[j,\infty]$, and
$\pi_\omega(j) < \pi_\omega(i)$ otherwise.

For $i<j$, let $\pi_\omega[i,j]$ denote the permutation of
$\{1,2,\ldots,j-i+1\}$ for which $\pi_\omega[i,j](k)<\pi_\omega[i,j](\ell)$
exactly when $\pi_\omega(i+k-1)<\pi_\omega(i+\ell-1)$.  If $j-i+1=n$
we say that the permutation $\pi_\omega[i,j]$ is a \emph{finite
  subpermutation of length $n$} of $\pi_\omega$.  The
\emph{permutation complexity} of $\omega$ is the function $f_\omega(n)$ that
associates every $n$ to the number of finite subpermutations of length
$n$ of $\pi_\omega$.

If $u$ is a factor of length $n$ of $\omega$, define
\[
\mbox{Perm}_\omega(u) = \{\pi_\omega[i,i+n-1] : \omega[i,i+n-1] = u\}.
\]
We say that $u$ has $|\mbox{Perm}_\omega(u)|$ permutations.
Furthermore, if $\mbox{Perm}_\omega(u) \cap \mbox{Perm}_\omega(v) \neq
\emptyset$, we say that $u$ and $v$ are \emph{factors with the same
  permutation}.

Our goal is to analyze the permutation complexity of the morphic image
of Sturmian words.  A \emph{Sturmian word} is an infinite word with
factor complexity $n+1$ for all $n \geq 0$ (the \emph{factor
  complexity} of an infinite word $w$ is the function giving the number of
distinct factors of $w$ of length $n$).  Let $s$ be a Sturmian
word over $\{0,1\}$ and let $T : \{0,1\} \to \{0,1\}$ be a morphism
such that $T(s)$ is aperiodic.  Then $T(s)$ has factor complexity
$n+t$ for some constant $t$ and all $n$ sufficiently large \cite{Pau75}.  Makarov
\cite{Mak09} showed that $f_s(n)=n$ for all $n \geq 2$.  We conjecture
that $f_{T(s)}(n) = n+k$ for some constant $k$ and all $n$
sufficiently large; however, we are only able to prove this for
``marked'' morphisms (defined below).

If the first letters of $T(0)$ and $T(1)$ are both different and the
last letters of $T(0)$ and $T(1)$ are both different, then we say that
$T$ is a \emph{marked morphism}.
If $T(0)$ and $T(1)$ are powers of a common word we say that $T$ is a
\emph{periodic morphism}; if not we say that $T$ is an \emph{aperiodic
  morphism}.  Note that a marked morphism is necessarily aperiodic.

A factor $u$ of an infinite word $s$ is \emph{right special} (resp.\
\emph{left special}) if both $u0$ and $u1$ (resp.\ $0u$ and $1u$) are
factors of $s$.  If $s$ is Sturmian then for all $n \geq 0$ the word $s$
contains exactly one right special factor of length $n$ and exactly one
left special factor of length $n$ (see \cite[Section~2.1.3]{Lot02}).  If
$T(s)$ is aperiodic then for $n$ sufficiently large the word $T(s)$
contains exactly one right special factor of length $n$ and exactly
one left special factor of length $n$.

An infinite word $s$ is \emph{uniformly recurrent} if for every
length $\ell$ there is another length $L$ such that every factor of
$s$ of length $L$ contains every factor of $s$ of length $\ell$.  If
$s$ is Sturmian, then $s$ and $T(s)$ are both uniformly recurrent.

\section{Recognizability of a morphism}
We also need some results concerning the recognizability of the
morphism $T$.  The basic definitions are given in terms of bi-infinite
words (following \cite{BSTY17}).

\begin{defn}
Let $\theta:A^* \to B^*$ be a non-erasing morphism; let $x = \cdots
x_{-1} x_0 x_1 \cdots$ be a bi-infinite word with each $x_i \in A$;
and let $y=\theta(x)$.  The set of \emph{cutting points} of
$(\theta,x)$ is the set
\[
C(\theta,x) = \{0\} \cup \{|\theta(x[0,i])| : i \geq 0\} \cup
\{-|\theta(x[-i,-1])| : i > 0\}.
\]
\end{defn}

\begin{defn}\label{recog_def}
Let $\theta:A^* \to B^*$ be a non-erasing morphism; let $x \in
A^{\mathbb{Z}}$; and let $y=\theta(x)$.  The morphism $\theta$ is
\emph{recognizable in the sense of Moss\'e for $x$} if there exists
$\ell$ such that, for every $m \in C(\theta,x)$ and $m' \in
\mathbb{Z}$, the equality $y[m-\ell,m+\ell-1] = y[m'-\ell,m'+\ell-1]$
implies that $m' \in C(\theta,x)$.
\end{defn}

In the special case of binary morphisms we have the following.

\begin{lemma}\label{mosse}
Let $T : \{0,1\}^* \to \{0,1\}^*$ be an aperiodic morphism.  Then $T$
is recognizable in the sense of Moss\'e for any aperiodic word $x \in
\{0,1\}^\mathbb{Z}$.
\end{lemma}

\begin{proof}
This follows from \cite[Theorems~3.1 and 2.5(1)]{BSTY17}.
\end{proof}

\begin{defn}
Let $\theta:A^* \to B^*$ be a non-erasing morphism; let $x \in
A^{\mathbb{Z}}$; and let $w$ be a non-empty factor of $\theta(x)$.  An
\emph{interpretation of $w$ in $x$} is a triple $(p,z,s)$ such that
\begin{itemize}
\item $z=z_0\cdots z_{n-1}$ is a factor of $x$ (each $z_i \in A$),
\item $p$ is a proper prefix of $\theta(z_0)$,
\item $s$ is a proper suffix of $\theta(z_{n-1})$, and
\item $\theta(z)=pws$.
\end{itemize}
\end{defn}

\begin{lemma}\label{marked_interpretation}
Let $T : \{0,1\}^* \to \{0,1\}^*$ be an aperiodic, marked morphism; let $x \in
\{0,1\}^\mathbb{Z}$; let $\ell$ be the constant of Definition~\ref{recog_def};
and let $w$ be a factor of $T(x)$ of length at least
$L = 2\ell+\max\{|T(0)|,|T(1)|\}$. Then $w$ has a unique
interpretation in $x$.
\end{lemma}

\begin{proof}
Consider two occurrences of $w$ in $T(x)$.  In the first occurrence there
is some position $m$ such that $m$ is at distance at least $\ell$ from
both the beginning and the end of $w$ and is a cutting point.  By
Lemma~\ref{mosse} and Definition~\ref{recog_def}, the corresponding
position in the second occurrence of $w$ is also a cutting point.
Now, since $T$ is marked, the interpretations of both
occurrences are uniquely determined.
\end{proof}

\section{Permutation complexity of $T(s)$}
Let $s$ be a Sturmian word over $\{0,1\}$ and let $T : \{0,1\}^* \to
\{0,1\}^*$ be a marked morphism.  Let $\bar{s}$ denote the word
obtained from $s$ by applying the morphism $0 \to 1, 1 \to 0$, and let
$\bar{T}$ denote the morphism defined by $0 \to T(1), 1 \to T(0)$.
Note that $\bar{s}$ is again Sturmian and 
we have $T(s) = \bar{T}(\bar{s})$.  Hence, without loss of generality,
we suppose that $T(0)$ begins with $0$ and $T(1)$ begins with $1$
(replacing $T$ and $s$ with $\bar{T}$ and $\bar{s}$ if necessary).

\begin{theorem}\label{main_thm}
There exist constants $N$ and $k$ such that 
the permutation complexity of $T(s)$ is $n+k$ for $n>N$.
\end{theorem}

Plan of the proof:
\begin{itemize}
	\item Lemma \ref{rs_same_perm} handles distinct factors having the same permutation.
	\item Lemma \ref{min_dual_perm} shows that minimal factors with multiple permutations are small.
	\item Lemma \ref{three_perms} shows that other than small exceptions, factors of $T(s)$ have at most two permutations.
	\item Lemmas \ref{nondecreasing_dual_factors} and \ref{non_inc} show that the number of factors with two permutations is (eventually) constant.
\end{itemize}
Once these facts are proved, we conclude that (other than small
exceptions) there are exactly $l$ factors with two permutations of
length $n$ for each $n$, and that every factor of $T(s)$ has exactly
one or two permutations. The result follows with either $k=t+l$ or $k=t+l-1$ (depending on the result of Lemma \ref{rs_same_perm}) where $t$ is the integer such that $T(s)$ has $n+t$ factors of length $n$.

\begin{example}
Let $s$ be the Fibonacci word; i.e. $s$ is the fixed point of the
morphism $0 \to 01$, $1 \to 0$.  Let $T$ be the morphism that maps $0
\to 0110$ and $1 \to 11$.
For $n\geq 14$, the word $T(s)$ has exactly $10$ factors of length $n$
with two permutations.
\end{example}

First, we deal with distinct factors of $s$ having the same permutation. We start with some basic general results. 

We need the following important fact due to Makarov
\cite[Lemma~1]{Mak06}:  Let $u$ and $v$ be distinct factors of $s$ of
the same length (greater than $1$) that have the same permutation.  Then $u$
and $v$ differ only in the last position.

\begin{lemma}
	\label{rs_same_perm}
	In $T(s)$, for $n$ sufficiently large, there is at most one
        pair of distinct factors of length $n$ with the same
        permutation.  If there are such pairs for infinitely many $n$,
        then there are such pairs for all $n$.
\end{lemma}

\begin{proof}
Let $u$ and
$v$ be factors of $T(s)$ of length $n$ and suppose that $u$ and $v$
have the same permutation.  Then write $u=w0$ and $v=w1$; we see that $w$ is
right-special.  If $n$ is sufficiently large, the word $T(s)$ contains exactly one
right-special factor of length $n-1$, so there can be at most one such
pair $u, v$.  Now if $u$ and $v$ have the same permutation, then so do
any of their equal-length suffixes, so if there are such pairs for
infinitely many $n$, then there are such pairs for all $n$.
\end{proof}

\begin{defn}
Let $u$ and $v$ be finite words.  A morphism $T$ is
\emph{order-preserving} if whenever $u \leq v$ we have
$T(u) \leq T(v)$.  If the same holds true whenever $u$ and $v$ are
infinite words we say that $T$ is \emph{order-preserving on infinite
  words}.
\end{defn}

These morphisms are studied further in Section~\ref{sec:order_preserving}.
Since we have assumed that $T(0)$ starts with $0$ and
$T(1)$ starts with $1$, the morphism $T$ is order-preserving on infinite
  words.

We now examine when it is possible for a factor $w$ of length $n$
to have more than one distinct permutation. In this case there must exist two
occurrences of $w$ in $T(s)$---say at positions $i$ and $i'$---and
integers $\ell$ and $\ell'$ satisfying $0 \leq \ell, \ell' \leq n-1$ such that
$T(s)[i+\ell,\infty] < T(s)[i+\ell',\infty]$ and $T(s)[i'+\ell,\infty]
> T(s)[i'+\ell',\infty]$.  It follows that $T(s)[i+\ell] =
T(s)[i+\ell'] = x$, for some letter $x$.  There then must exist
factors $w_0$ and $w_1$ of $T(s)$, each with prefix $w$, having the
following forms:
\begin{align}
\label{std_dual_perm}
w_0 &= \rho_1 x u 0 \gamma = \rho_2 x u 1\\
w_1 &= p_1  x v 1 g = p_2 x v 0, \nonumber
\end{align}
for some words $\rho_1, \rho_2, \gamma, p_1, p_2, g$, where $|\rho_1|
< |\rho_2$ and $|p_1| < |p_2|$,
the common prefix $w$ extends at least to include
the second $x$, and the $x$'s have the same relative indices in $w_0$
and $w_1$. Let us assume that $|u|>|v|$.  We need the following result
\cite[Lemma~3]{Mak09}:

\begin{lemma}[\cite{Mak09}]\label{makarov}
Let $z$ be a factor of length $n+1$ of a Sturmian word $s$.  Then $z$
has exactly one permutation and this permutation is uniquely
determined by the prefix of $z$ of length $n$.
\end{lemma}

We also need the following well-known result about repetitions in
Sturmian words.

\begin{lemma}\label{exponent}
Let $s$ be a Sturmian word and let $T$ be an aperiodic binary morphism.
For any integer $p \geq 1$ there is a constant $K_0(p)$ (resp.\
$K(p)$) such that every factor of $s$ (resp.\ $T(s)$) of period at
most $p$ has length at most $K_0(p)$ (resp.\ $K(p)$).
\end{lemma}

\begin{proof}
The claim is an easy consequence of the fact that $s$, and hence
$T(s)$, is aperiodic but uniformly recurrent.  Recall that this means that for every
length $\ell$ there is another length $L$ such that every factor of
$s$ of length $L$ contains every factor of $s$ of length $\ell$.  If,
contrary to the claim, there were unboundedly large factors of $s$ of
period $p$, these factors would necessarily fail to contain some factor
of $s$ of length $p$.
\end{proof}

In the rest of the argument, we will often wish to indicate that
certain types of factors have lengths that are bounded by some
absolute constant depending only on $T$ and $s$.  We will abbreviate
this notion by saying that these factors are \emph{small}.

\begin{lemma}
	\label{dual_perm_small}
	In Equation~\eqref{std_dual_perm}, the words $u$ and $v$ are small.
\end{lemma}

\begin{proof}
Let $P$ be the longest common prefix of $w_0$ and
$w_1$. The assumption $|u|>|v|$ implies that $P$ has suffix
$v$.   Let $L$ be the constant of Lemma~\ref{marked_interpretation}
(where the $x$ of the lemma is any aperiodic extension of $s$ to a
bi-infinite word).  If $|xu|<L$ then $u$ and hence $v$ are small and
we are done, so suppose instead that $|xu| \geq L$.

Suppose first that $|P| \geq L$.  Then by
Lemma~\ref{marked_interpretation}, the words $xu$ and $P$ each have
unique interpretations in $s$.  Let $(\pi,z,\sigma)$ be the
interpretation of $P$ in $s$.  Then there exist positions
$I, J$ in $z$ such that the two $x$'s in $w_0$ and $w_1$ occur in
$T(z[I])$ and $T(z[J])$.  Furthermore, by the uniqueness of the
interpretations of $xu$ in $s$ we have $z[I]=z[J]$ and the $x$'s occur
at the same positions of $T(z[I])$ and $T(z[J])$.  Recalling that $T$
is order-preserving on infinite words, we see that by
Lemma~\ref{makarov}, the relative
orders of both pairs $(T(s)[i+\ell,\infty], T(s)[i+\ell',\infty])$ and
$(T(s)[i'+\ell,\infty], T(s)[i'+\ell',\infty])$ are determined by the
factor $z$ of $s$.  This contradicts the assumption that these two
pairs of infinite words have opposite relative orders.

Now suppose that $|P|<L$.  Since $|xu| \geq L > |P|$ and $P$ contains
both occurrences of $x$, the two occurrences of $xu$ in $w_0$ must
overlap.  Let $Q$ be the factor of $w_0$ consisting of exactly these
two overlapping occurrences of $xu$.  The word $Q$ has a period which
is at most the distance between the two $x$'s, and since $P$ contains both
$x$'s, this period is therefore at most $|P|$.  Then by
Lemma~\ref{exponent} we have $|Q| \leq K(|P|)$, and so \textit{a
  fortiori}, we have $|u| \leq K(|P|)$.

In both cases, we get an upper bound on the lengths of both $u$ and
$v$ that depends only on $T$ and $s$.
\end{proof}

In the next lemma, by \emph{minimal} we mean that no proper factor has
two permutations.

\begin{lemma}
	\label{min_dual_perm}
	Minimal factors of $T(s)$ with two permutations are small.
\end{lemma}

\begin{proof}
Let $w$ be a minimal factor of $T(s)$ with two permutations.
Let $w_0$ and $w_1$ be as in Equation~\eqref{std_dual_perm} with
$|u|>|v|$.  By the minimality of $w$ we may suppose that $w_0$ and
$w_1$ begin with the first occurrence of $xu0$. 

By Lemma~\ref{dual_perm_small}, the word $u$ is small.  If $w$ is
sufficiently large, then (by the uniform recurrence of $T(s)$) it contains
a second occurrence of $xu0$ (which begins with $xv1$), which
contradicts the minimality of $w$.
\end{proof}

Next, we handle factors with more than two permutations.

\begin{lemma}
	\label{three_perms}
	Factors of $T(s)$ with more than two permutations are small.
\end{lemma}

\begin{proof}
	Suppose $w$ has three permutations in $T(s)$. Let
        $w_0,w_1,w_2$ be minimal length factors of $T(s)$ with prefix $w$
        extended far enough to the right for the permutations of $w$
        to be determined. Assume also that $w$ has a different
        permutation in each. Suppose further that the longest common
        prefix of $w_0$ and $w_1$ is shorter than the longest common
        prefix of $w_0$ and $w_2$.  As in Equation \ref{std_dual_perm}, we may write 	
	\begin{align*}
w_0 &= \rho_1 x u 0 \gamma = \rho_2 x u 1\\
w_1 &= p_1  x v 1 g = p_2 x v 0, 
	\end{align*} where $|u|>|v|$. Thus
        $v1$ is a prefix of $u$, and thus the common prefix of $w_0,
        w_1$ ending in the second $xv$ (call this $P_0$) is right
        special. By Lemma \ref{dual_perm_small}, $|P_0|-|w|$ is small.
        Similarly, since $w$ has different permutations in $w_0$ and
        $w_2$, these two words have a common right special prefix
        $P_1$, where again $|P_1|-|w|$ is small.  Now if $w$ is
        large, then so are $P_0$ and $P_1$, and hence $T(s)$ contains
        exactly one right special factor of length $|P_0|$.
        Consequently, the suffix of $P_1$ of length $|P_0|$ is in fact
        equal to $P_0$.  Since $|P_0|-|w|$ and $|P_1|-|w|$ are both small, the quantity
        $|P_1|-|P_0|$ is also small.  It follows that $P_1$ has period
        $|P_1|-|P_0|$, and therefore, by Lemma~\ref{exponent}, we have
        $|P_1| \leq K(|P_1|-|P_0|)$, as required.
\end{proof}

\begin{lemma}
	\label{nondecreasing_dual_factors}
	If $T(s)$ has $k$ factors of length $n$ with two permutations
        for $n$ sufficiently large, then $T(s)$ has at least $k$
        factors of length $n+1$ with two permutations.
\end{lemma}

\begin{proof}
Suppose that $n$ is sufficiently large that $T(s)$ has exactly one
right special factor of each length for lengths $n$ and larger.  Let
$w$ be a factor of length $n$ with two permutations in $T(s)$. Note
that we can uniquely extend $w$ to the right until the result becomes
a right special factor of $T(s)$. As in Equation~\ref{std_dual_perm},
write 
	\begin{align*}
w_0 &= \rho_1 x v1q 0 \gamma = \rho_2 x v1q 1\\
w_1 &= p_1  x v 1 g = p_2 x v 0,
	\end{align*} where $v1q = u$. Let $P_0$ be the common prefix
        of $w_0,w_1$ ending with the second $v$. Note that $P_0$ is right
        special. If both $aP_01q1$ and $aP_00$ occur in $T(s)$ for
        $a=0$ or $a=1$, then $aw$ has two permutations. Assume this is
        not the case. Let $a$ be such that $aP_0$ is right special
        (such an $a$ exists, since otherwise there would be two right
        special factors of length $|P_0|$). Then any occurrence of
        $aP_01$ is not followed by $q1$. Hence, $P_01y$ is right
        special for $y$ some prefix of $q1$. By
        Lemma~\ref{dual_perm_small}, $v$ and $q$ (and thus $y$) are
        small. Set $P_1 = P_01y$ and apply the same argument as in the end
        of the proof of Lemma \ref{three_perms}.  We find that $|P_1|
        \leq K(|P_1|-|P_0|)$, contradicting the assumption that $w$ is
        large.
\end{proof}

\begin{lemma}
	\label{non_inc}
	If $T(s)$ has $k$ factors of length $n$ with two permutations
        for $n$ sufficiently large, then $T(s)$ has at most $k$
        factors of length $n+1$ with two permutations.
\end{lemma}

\begin{proof}
  Suppose that $n$ is sufficiently large that $T(s)$ has exactly one
  right special factor of each length for lengths $n$ and larger.  If
  $T(s)$ has no factors with two permutations the result is trivial,
  so assume otherwise. Let $aw$ be a factor of $T(s)$ of length $n+1$
  with two permutations where $|a|=1$. If $w$ does not have two
  permutations, then the same argument as in the proof of
  Lemma~\ref{min_dual_perm} applied to $aw$. where $a$ necessarily
  plays the role of $x$ (since otherwise, if the $x$'s were contained
  in $w$, then $w$ would have two permutations), shows that in this
  case $aw$ is small, which is a contradiction.  So in fact $w$ does
  have two permutations.

This shows that there are at most $k$ factors of length $n+1$ with two
permutations except in one particular circumstance: $w$ is left
special and both $aw$ and $bw$ have two permutations, where $a$ and
$b$ are different letters.  Write
	\begin{align*}
w_0 &= a\rho_1 x vcq d \gamma = \rho_2 x vcq c\\
w_1 &= ap_1  x v c g = p_2 x v d,
	\end{align*} 
where $x \in \{0,1\}$, $c$ and $d$ are different letters, and the
relative positions of the $x$'s in $w_0$ and $w_1$ are the same.
Similarly, write
	\begin{align*}
w_0' &= b\rho_1' x' v'c'q' d' \gamma' = \rho_2' x' v'c'q' c'\\
w_1' &= bp_1'  x' v' c' g' = p_2' x' v' d',
	\end{align*} 
where $x' \in \{0,1\}$, $c'$ and $d'$ are different letters, and the
relative positions of the $x'$'s in $w_0'$ and $w_1'$ are the same.

Let $aP_0$ (resp.\ $bP_1)$
be the longest common prefix of $w_0$ and $w_1$ (resp.\ $w_0'$ and
$w_1'$).  Then $aP_0$ ends with the second $xv$ and $bP_1$ ends with
the second $x'v'$.   Furthermore, both $P_0$ and $P_1$ are right
special and both have $w$ as a prefix.  If $|P_0|=|P_1|$, then $aP_0$
and $bP_1$ are distinct right special factors of the same length,
which is a contradiction.  So suppose that $|P_1| > |P_0|$.  As in the
end of the proof Lemma~\ref{three_perms}, we argue that since $T(s)$ only
contains one right special factor of length $|P_0|$, the suffix of
$P_1$ is equal to $P_0$.  However, unlike in Lemma~\ref{three_perms},
we cannot say that $P_0$ is also a prefix of $P_1$.  So let $P_1'$ be
the prefix of $P_1$ of length $|P_1|-|P_0|+|w|$.  Then $P_1'$ begins
and ends with $w$.  Consequently, $P_1'$ has period $|P_1'|-|w|$.  By
Lemma~\ref{dual_perm_small}, the quantity $|P_1|-|w|$ and hence
$|P_1'|-|w|$ is small.  Applying Lemma~\ref{exponent}, we find that
$|P_1'| \leq K(|P_1'|-|w|)$, which is again a contradiction.
\end{proof}

This completes the proof of Theorem~\ref{main_thm}.

\section{Effect of an arbitrary aperiodic morphism on the order}
\label{sec:order_preserving}
Clearly, one would like to show that Theorem~\ref{main_thm} holds
without the assumption that $T$ is marked.  There are two
difficulties: first, we need to establish that $T$ always preserves
(or reverses) the order on infinite words; and second, we need
recognizability properties similar to the marked case.  The latter
issue seems difficult to resolve, but we can establish the
necessary properties regarding the order.

Richomme \cite[Lemma~3.13]{Ric03} characterized the
order-preserving binary morphisms.

\begin{lemma}[\cite{Ric03}]\label{order-preserving}
Let $T : \{0,1\}^* \to \{0,1\}^*$ be a non-empty morphism.  Then $T$ is
order-preserving if and only if $T(01) < T(1)$.
\end{lemma}

Let $u$ and $v$ be infinite words.  A morphism $T$ is
\emph{order-reversing on infinite words} if whenever $u < v$
we have $T(u) > T(v)$.

\begin{lemma}\label{order-reversing}
Let $T : \{0,1\}^* \to \{0,1\}^*$ be a morphism such that
$T(1)$ is not a prefix of $T(01)$.  Suppose that $T(01) > T(1)$.  Then
$T$ is order-reversing on infinite words.
\end{lemma}

\begin{proof}
By hypothesis $T(1)$ is not a prefix of $T(01)$, so $T(0) \neq \epsilon$,
and therefore we can write $T(01) = X1Y$ and $T(1) = X0z$ for some
words $X$, $Y$, and $z$.  Let $k$ be maximal such that
$T(0^k)$ is a prefix of $T(1)$.  Then $T(0^k)$ is a prefix of $T(01)$ and hence a
prefix of $X$.  Write $X = T(0^k)x$.  Then $T(1) = T(0^k)x0z$ and
$T(01) = T(0^{k+1})x0z = T(0^k)x1Y$.  Hence $T(0)x0z = x1Y$.  By the
maximality of $k$, the word $T(0)$ is not a prefix of $x$, so $x1$ is
a prefix of $T(0)$.  We therefore have $T(0) = x1y$ for some word $y$
and $T(1) = T(0^k)x0z$.

Now let $u$ and $v$ be infinite words such that $u<v$.  Without loss
of generality we may assume that $u$ begins with $0$ and $v$ begins
with $1$.  Then $T(u)$ begins with $T(0^{k+1}) = T(0^k)x1y$ and $T(v)$
begins with $T(1) = T(0^k)x0z$, and so $T(u) > T(v)$.  Hence $T$ is
order-reversing on infinite words, as required.
\end{proof}

\begin{lemma}\label{T1prefT01}
Let $T : \{0,1\}^* \to \{0,1\}^*$ be an aperiodic morphism such that
$T(1)$ is a prefix of $T(01)$.  Then $T$ is either
order-preserving on infinite words or $T$ is order-reversing on
infinite words.
\end{lemma}

\begin{proof}
Case~1: $T(1)$ is not a prefix of $T(0)$.
Note that $T(1)$ has period $|T(0)|$, and so we can write $T(0) = x^k$
and $T(1) = x^\ell y$, where $x$ is primitive, the exponent $\ell$ is
maximal, and $y$ is a non-empty proper prefix of $x$.  Let $u$ and $v$ be
infinite words with $u<v$.  Without loss of generality we may assume
that $u$ begins with $0$ and $v$ begins with $1$.  Note that $T(u)$
begins with $x^{\ell+1}y$ (since $y$ is a prefix of $x$) and $T(v)$
begins with $x^{\ell}yx$.  If $x^{\ell+1}y \neq x^{\ell}yx$, then
either $T(u) < T(v)$ for all $u<v$ or $T(u) > T(v)$ for all $u<v$.  Thus $T$ is either
order-preserving  on infinite words or order-reversing on infinite words.  If
$x^{\ell+1}y = x^{\ell}yx$, then  we have $xy=yx$, which implies that
$x$ and $y$ are powers of a common word, which contradicts the
primitivity of $x$.

Case~2: $T(1)$ is a prefix of $T(0)$.
Write $T(1) = x^k$ and $T(0) = x^{\ell}y$, where $x$ is primitive,
the exponent $\ell$ is maximal, and $y$ is non-empty.  Let $u$ and
$v$ be infinite words with $u<v$.  Without loss of generality we may
assume that $u$ begins with $0$ and $v$ begins with $1$.  Note that
$T(u)$ begins with $x^{\ell}y$ and $T(v)$ begins with
$x^{\ell+1}$.  If $y$ is not a prefix of $x$ then either $T(u) <
T(v)$ for all $u<v$ or $T(u) > T(v)$ for all $u<v$.
Thus $T$ is either order-preserving  on infinite words or
order-reversing on infinite words.   If $y$ is a prefix of $x$, we
apply the argument from Case~1 to obtain the desired conclusion.
\end{proof}

\begin{theorem}\label{dichotomy}
Let $T : \{0,1\}^* \to \{0,1\}^*$ be an aperiodic morphism.  Then $T$ is either
order-preserving on infinite words or $T$ is order-reversing on
infinite words.
\end{theorem}

\begin{proof}
For an arbitrary morphism $T$, one of the following properties must
hold:
\begin{enumerate}
\item $T(01) < T(1)$ (in which case $T$ is order-preserving by
  Lemma~\ref{order-preserving});
\item $T(01) > T(1)$, but $T(1)$ is not a prefix of $T(01)$ (in which
  case $T$ is order-reversing on infinite words by Lemma~\ref{order-reversing});
\item $T(1)$ is a prefix of $T(01)$ and $T$ is aperiodic (in which
  case $T$ is either order-preserving on infinite words or
  order-reversing on infinite  words by Lemma~\ref{T1prefT01});
\item $T(1)$ is a prefix of $T(01)$ and $T$ is periodic.
\end{enumerate}
The only case where we don't have the desired conclusion is when $T$
is periodic.
\end{proof}

\section*{Acknowledgments}
We would like to thank Lucas Mol for his very helpful feedback on
earlier drafts of this paper.  We also thank the anonymous referees
for their excellent and useful comments.

\end{document}